\documentclass[a4paper,11pt]{article}
\usepackage{mathtools}
\usepackage{amsfonts}
\usepackage{amsthm}
\usepackage{amsbsy}
\usepackage{enumerate}
\usepackage{bbm}
\usepackage{fullpage}
\usepackage{color}
\usepackage{microtype}
\usepackage{hyperref}
\usepackage{todonotes}

\theoremstyle{plain}
\newtheorem{thm}{Theorem}
\newtheorem{lemma}[thm]{Lemma}
\newtheorem{prop}[thm]{Proposition}

\theoremstyle{definition}
\newtheorem{defi}[thm]{Definition}

\newtheorem{rmk}[thm]{Remark}

\begin{document}

\title{Attracting measures}
\author{Julian Newman and Peter Ashwin\\
Department of Mathematics and Statistics, \\
University of Exeter, \\
Exeter EX4 4QF, United Kingdom}

\maketitle

\begin{abstract}
\noindent Under mild assumptions, the SRB measure $\mu$ associated to an Axiom~A attractor $A$ has the following properties: (i)~the empirical measure starting at a typical point near $A$ converges weakly to $\mu$; (ii)~the pushforward of any Lebesgue-absolutely continuous probability measure supported near $A$ converges weakly to $\mu$. In general, a measure with the first property is called a ``physical measure'', and physical measures are recognised as generally important in their own right. In this paper, we highlight the second property as also important in its own right, and we prove a result that serves as a topological abstraction of the original result that establishes the second property for SRB measures on Axiom A attractors.
\end{abstract}

\section{Introduction} \label{sec:Intro}

%\subsection{Two properties of SRB measures on Axiom A attractors}

Every compact invariant set of a continuous autonomous dynamical system supports at least one invariant probability measure, and often supports uncountably many invariant probability measures. An important question in ergodic theory is whether one of the invariant measures on a given compact invariant set can be singled out as ``the natural'' measure on that set, in that this measure in some sense describes the typical behaviour of trajectories in a neighbourhood of the set.

An Axiom~A attractor $A$ of a $C^2$-diffeomorphism $f$ in discrete time~\cite{Ruelle:1976}, or of a flow $f\!=\!(f^t)_{t \in \mathbb{R}}$ generated by a $C^2$ vector field in continuous time~\cite{Bowen:1975}, is equipped with such a ``natural'' probability measure $\mu$ of full support in $A$, which has come to be known as the \emph{SRB measure} on $A$. The original papers \cite{Ruelle:1976,Bowen:1975} highlight two such ``naturalness'' properties of the SRB measure---namely, letting $m_A$ denote the Lebesgue measure restricted to the basin of attraction of $A$:
\begin{itemize}
    \item The first property concerns the almost-sure sample statistics of a random trajectory starting near $A$. It says that, in a certain sense, the nonsingular dynamical system $(f,m_A)$ behaves like a $\mu$-ergodic dynamical system.
    \item The second property concerns the limiting distribution of a random trajectory starting near $A$. It says that, in a certain sense, the nonsingular dynamical system $(f,m_A)$ behaves like a $\mu$-mixing dynamical system.
\end{itemize}
The first property holds with no extra conditions beyond the general setup of the papers \cite{Ruelle:1976,Bowen:1975}. The second property holds in discrete time as long as $A$ is connected, and holds for flows under a mild condition, e.g.\ that the full unstable manifold of each point in $A$ is dense in $A$.

Although each of the original papers \cite{Ruelle:1976,Bowen:1975} gives somewhat similar prominence to both of the above two properties, subsequent expository literature and further research have tended to focus more on the first property, giving it a fairly standard name (``\emph{physical measure}'') and exploring it in greater generality than the Axiom~A setting. However, just as the fundamental importance of ergodicity in dynamical systems does not cause the importance of the study of mixing to be unrecognised, so likewise the attention received by the topic of ``physical measures'' should not be to the detriment of the study of measures with the second property. In fact, for invariant measures that are singular with respect to the Lebesgue measure, this second property is more accessible~experimentally than classical mixing is~\cite{Baladi:2002}.

In this paper, we will refer to measures satisfying the second property as \emph{attracting measures}. We have introduced this terminology in the papers~\cite{ashwin2021physical,newman23} particularly in the context of dynamical systems subject to a real-time parameter drift, but our present paper will focus on the more fundamental setting of autonomous dynamical systems. We will present some fundamental facts about attracting measures, and then in Theorem~\ref{thm:mixSRB}, we provide a topological-level abstraction of the strategy of proof that was used in \cite[Theorem~5.3]{Bowen:1975} for establishing the attracting-measure property of the SRB measure $\mu$ on an Axiom~A attractor of a $C^2$ flow. Essentially, the strategy in \cite[Theorem~5.3]{Bowen:1975} was to use the fact established earlier in the paper that $\mu$ is mixing, and combine this with two other earlier-established facts about the local dynamics near the attractor: crudely speaking, one of these is that trajectories near the attractor track trajectories on the attractor, and the other is that the ``chaoticness'' on the attractor is ``not statistically much greater'' according to $\mu$ than according to the Lebesgue measure.

The structure of the paper is as follows. In the remainder of Sec.~\ref{sec:Intro}, we present some general notations, terminologies, and lemmas from measure theory. In Sec.~\ref{sec:Basic}, we start by stating the general setup of the whole paper; we then introduce definitions of ergodic, physical, mixing and attracting measures, and discuss their basic properties and mutual implications and non-implications. In Sec.~\ref{sec:Result}, we define a type of attractor that represents one of the ways in which trajectories near Axiom~A attractors track trajectories on the attractor; and with this, we then state and prove Theorem~\ref{thm:mixSRB}. Finally, in Sec.~\ref{sec:outlook}, we discuss natural further directions for the work.

\subsection{Some general preliminaries}

By way of notation:
\begin{itemize}
%\item We write ``$G_t \to K$ as $t \to \infty$'', where each $G_t$ is a non-empty closed subset of a metric space $X$ and $K$ is a non-empty compact subset of $X$, to mean that the Hausdorff distance between $G_t$ and $K$ tends to $0$ as $t \to \infty$.
\item For a semiflow $(f^t)_{t \geq 0}$ of self-maps of a set $X$, given any $t \geq 0$ and $S \subset X$ we write $f^{-t}S$ for the preimage of $S$ under $f^t$.
\item For a semiflow $(f^t)_{t \geq 0}$ of measurable self-maps of a measurable space $(X,\mathcal{B})$, given any $t \geq 0$ and any measure $\nu$ on $(X,\mathcal{B})$, we write $f^t\nu$ for the pushforward measure of $\nu$ by $f^t$, that is,
\[ (f^t\nu)(S) = \nu(f^{-t}S) \quad \forall \, S \in \mathcal{B}. \]
\end{itemize}

Recall that on a measurable space $(X,\mathcal{B})$, a family $(\mu_t)_{t \geq 0}$ of finite measures $\mu_t$ is said to \emph{converge strongly} to a finite measure $\mu$ if the following equivalent statements hold:
\begin{itemize}
    \item for every bounded measurable function $g \colon X \to \mathbb{R}$, we have $\int g \, d\mu_t \to \int g \, d\mu$ as $t \to \infty$;
    \item for every $A \in \mathcal{B}$, we have $\mu_t(A) \to \mu(A)$ as $t \to \infty$.
\end{itemize}

Recall that on a separable metric space $(X,d)$, a family $(\mu_t)_{t \geq 0}$ of finite Borel measures $\mu_t$ is said to \emph{converge weakly} to a finite Borel measure $\mu$ if the following equivalent statements hold:
\begin{itemize}
    \item for every bounded continuous function $g \colon X \to \mathbb{R}$, we have $\int g \, d\mu_t \to \int g \, d\mu$ as $t \to \infty$;
    \item for every bounded Lipschitz function $g \colon X \to \mathbb{R}$, we have $\int g \, d\mu_t \to \int g \, d\mu$ as $t \to \infty$;
    \item for every $A \in \mathcal{B}(X)$ with $\mu(\partial A)=0$, we have $\mu_t(A) \to \mu(A)$ as $t \to \infty$;
    \item $\mu_t(X) \to \mu(X)$ as $t \to \infty$ and for every open $U \subset X$, we have $\mu(U) \leq \liminf_{t \to \infty} \mu_t(U)$.
\end{itemize}

A `neighbourhood' of a set $S \subset X$ will mean a set containing an open set containing $S$.

\subsubsection{A preliminary lemma}

Given a measure space $(X,\mathcal{B},m)$,
\begin{itemize}
    \item for a set $C \subset L^1(m)$, we write $\langle C \rangle \subset L^1(m)$ for the additive semigroup generated by $C$,
    \item we write $L_+^1(m) := \{h \in L^1(m) \, : \, h \geq 0\}$,
    \item a \emph{subcone of $L_+^1(m)$} means a set $C \subset L_+^1(m)$ such that for every $h \in L_+^1(m)$ and $c>0$, $ch \in L_+^1(m)$.
\end{itemize}

\begin{lemma} \label{lemma:cone_conv}
Suppose we have a $\sigma$-finite measure space $(X,\mathcal{B},m)$, a probability measure $\mu$ on $(X,\mathcal{B})$, a subcone $C$ of $L_+^1(m)$ with $\langle C \rangle$ being dense in $L_+^1(m)$, a family $(f_t)_{t \geq 0}$ of measurable functions $f_t \colon X \to X$, and a bounded measurable function $g \colon X \to \mathbb{R}$, such that for every $m$-absolutely continuous probability measure $\nu$ with $\frac{d\nu}{dm} \in C$,
\begin{equation} \label{eq:lemma_conv} \int g \, d(f_t\nu) \to \int g \, d\mu \quad \textrm{as } t \to \infty. \end{equation}
Then Eq.~\eqref{eq:lemma_conv} holds for every $m$-absolutely continuous probability measure $\nu$ (even when $\frac{d\nu}{dm} \not\in C$).
\end{lemma}

\begin{proof}

Note that $\langle C \rangle$ is itself a subcone of $L_+^1(m)$. For any $m$-absolutely continuous probability measure $\nu$ with $\frac{d\nu}{dm} \in \langle C \rangle$, writing
\[ \nu = \sum_{i=1}^n \nu_i \]
where the measures $\nu_i$ are $m$-absolutely continuous finite measures with $\nu_i(X)>0$ and $\frac{d\nu_i}{dm} \in C$, we have
\begin{align*}
   \int g \, d(f_t\nu)
   =  \sum_{i=1}^n \int g \, d(f_t\nu_i)
   = \; \sum_{i=1}^n \nu_i(X) \int g \, d\big(f_t\big(\tfrac{\nu_i}{\nu_i(X)}\big)\big)
   &\to \sum_{i=1}^n \nu_i(X) \int g \, d\mu \quad \textrm{by \eqref{eq:lemma_conv}} \\
   &= \nu(X) \int g \, d\mu = \int g \, d\mu.
\end{align*}
So we can assume without loss of generality that $C$ is itself dense in $L_+^1(m)$.

Hence the set $\{h \in C \, : \, \int h \, dm=1\}$ is dense in $\{h \in L_+^1(m) \, : \, \int h \, dm=1\}$, since for any $h \in L_+^1(m)$ with $\int h \, dm=1$ and any sequence $(h_n)$ in $C$ that is $L^1$-convergent to $h$, we have that $\int h_n \, dm \to 1$ and so $\frac{1}{\int h_n \, dm} h_n$ converges in $L^1$ to $h$.

Now fix any $m$-absolutely continuous probability measure $\nu$, and let $h=\frac{d\nu}{dm}$. Fix $\varepsilon>0$. We can find an $m$-absolutely continuous probability measure $\tilde{\nu}$ with $\tilde{h}\!:=\!\frac{d\tilde{\nu}}{dm} \in C$ and $\int_U |h-\tilde{h}| \, dm < \frac{\varepsilon}{2\|g\|_\infty}$. For sufficiently large $t$, we have
\[ \left|\int g \, d(f_t\tilde{\nu}) - \int g \, d\mu \right| \ < \ \tfrac{\varepsilon}{2} \]
and so
\begin{align*}
    \left|\int g \, d(f_t\nu) - \int g \, d\mu \right| \ &\leq \ \left|\int g \, d(f_t\nu) - \int g \, d(f_t\tilde{\nu}) \right| + \left|\int g \, d(f_t\tilde{\nu}) - \int g \, d\mu \right| \\
    &= \ \ \left|\int (g \circ f_t).(h-\tilde{h}) \, dm \right| \; \, + \; \, \left|\int g \, d(f_t\tilde{\nu}) - \int g \, d\mu \right| \\
    &< \ \|g\|_\infty\tfrac{\varepsilon}{2\|g\|_\infty} + \tfrac{\varepsilon}{2} \\
    &= \ \varepsilon. \qedhere
\end{align*}

\end{proof}

\section{Setup, definitions, and basic results} \label{sec:Basic}

\subsection{General setup of this paper} \label{sec:general_setup}

We will formulate our definitions and results for semiflows; these can all easily be adapted to the case of discrete-time maps.

Throughout this paper: $(X,d)$ is a separable metric space; $m$ is a locally finite Borel measure on $X$ with full support in $X$; and $(f^t)_{t \geq 0}$ is a semiflow of maps $f^t \colon X \to X$ such that $(t,x) \mapsto f^tx$ is continuous and $X$ can be covered by open sets $U$ with the property that
    \begin{itemize}
        \item[(H)] for each $t \geq 0$, the measure $S \mapsto m(U \cap f^{-t}S)$ is $m$-absolutely continuous with bounded density.
    \end{itemize}
Note that a finite union of sets satisfying (H) also satisfies (H), since $m\big((U \cup V) \cap f^{-t}S\big) \leq m(U \cap f^{-t}S) + m(V \cap f^{-t}S)$.

Given $x \in X$, we write $B_\delta(x)$ for the open $\delta$-ball around $x$, and given $A \subset X$, we write $B_\delta(A)=\bigcup_{x \in A}B_\delta(x)$.

The above general setup is fulfilled by and generalises the following scenario: $X$ is a finite-dimensional Riemannian manifold, with $d$ the geodesic distance function and $m$ the Riemannian volume measure; and $(f^t)_{t \geq 0}$ is a semiflow of $C^1$ local diffeomorphisms $f^t \colon X \to X$ such that $(t,x) \mapsto f^tx$ is continuous. In this scenario, every open set with compact closure fulfils (H).

\begin{lemma} \label{lemma:nonsing}
For every $m$-absolutely continuous measure $\nu$ and every $t \geq 0$, $f^t\nu$ is $m$-absolutely continuous.
\end{lemma}

\begin{proof}
Every open cover of a separable metric space has a countable subcover, so let $\{U_n\}_{n \geq 1}$ be a countable open cover of $X$ by sets fulfilling (H). For any $m$-null set $S$ we have $f^t\nu(S) = \nu(f^{-t}S) \leq \sum_{n=1}^\infty \nu(U_n \cap f^{-t}S)=0$, since $m(U_n \cap f^{-t}S)=0$ for each $n$.
\end{proof}

\begin{lemma} \label{lemma:bd}
Let $U \subset X$ be an open set fulfilling \emph{(H)}. Then for any $m$-absolutely continuous measure $\nu$ with bounded density such that $\nu(X \setminus U)=0$, for every $t \geq 0$, $f^t\nu$ is $m$-absolutely continuous with bounded density.
\end{lemma}

\begin{proof}
Let $C > 0$ be such that for all $S \in \mathcal{B}(X)$, $\nu(S) \leq C\mu(S)$. Fix $t \geq 0$, and let $D > 0$ be such that for all $S \in \mathcal{B}(X)$, $m(U \cap f^{-t}S) \leq Dm(S)$. Then for all $S \in \mathcal{B}(X)$,
\[ f^t\nu(S) = \nu(f^{-t}S) = \nu(U \cap f^{-t}S) \leq Cm(U \cap f^{-t}S) \leq (CD)m(S). \qedhere \]
\end{proof}

\subsection{Ergodic measures and physical measures}

Throughout this paper, an \emph{invariant measure} means a Borel probability measure $\mu$ on $X$ such that for all $t \geq 0$, $f^t\mu=\mu$.

\begin{defi}
The \emph{basin} of a Borel probability measure $\mu$ on $X$ is given by
\[ \mathrm{Basin}(\mu) = \left\{ x \in X \, : \, \frac{1}{T} \int_0^T \delta_{f^tx}(\cdot) \, dt \, \to \, \mu \ \textrm{ weakly as $T \to \infty$} \right\}. \]
\end{defi}

\begin{rmk} \label{rmk:inv}
It is not hard to show that if $\mathrm{Basin}(\mu) \neq \emptyset$ then $\mu$ is an invariant measure.
\end{rmk}

Now there are about a dozen important equivalent definitions of an ergodic measure. A commonly given definition (applicable to any measurable semiflow on a measurable space) would be the following.

\begin{defi}
An \emph{ergodic measure} is an invariant measure $\mu$ such that for any $A \in \mathcal{B}(X)$ with the property that $f^{-t}A = A$ for all $t \geq 0$, we have $\mu(A) \in \{0,1\}$.
\end{defi}

For measurable semiflows on separable metric spaces (such as in our general setup), one of the equivalent definitions of ergodicity---perhaps the closest reflection of the original conception of ergodicity in thermodynamics---is as follows (cf.~\cite[Exercise~4.1.5]{VO:16}):

\begin{prop} \label{prop:ergequiv}
A Borel probability measure $\mu$ on $X$ is an ergodic measure if and only if $\mu$-almost every point in $X$ is in $\textrm{Basin}(\mu)$.
\end{prop}

\begin{proof}
For a separable metric space, weak convergence can be determined by countable family of bounded continuous functions $\{g_n\}_{n \in \mathbb{N}}$. Suppose $\mu$ is ergodic; then Birkhoff's Ergodic Theorem gives that for $\mu$-almost all $x \in X$, we have $\frac{1}{T}\int_0^T g_n(f^tx) \, dt \to \int_X g_n \, d\mu$ as $T \to \infty$ for all $n$, and hence $x \in \mathrm{Basin}(\mu)$. Now conversely, suppose that $\mu$-almost every point in $X$ is in $\textrm{Basin}(\mu)$. So by Remark~\ref{rmk:inv}, $\mu$ is invariant. Fix a set $A \in \mathcal{B}(X)$ such that $f^{-t}A=A$ for all $t \geq 0$. Define the finite measures $\mu_1$ and $\mu_2$ on $X$ by
\begin{align*}
    \mu_1(S) &= \mu(A)\mu(S) \\
    \mu_2(S) &= \mu(A \cap S)
\end{align*}
for all $S \in \mathcal{B}(X)$. Then for every bounded continuous function $g \colon X \to \mathbb{R}$ we have
\begin{align*}
    \int_X g \, d\mu_1 &= \mu(A)\int_X g(y) \, \mu(dy) \\
    &= \int_A \left( \int_X g(y) \, \mu(dy) \right) \, \mu(dx) \\
    &= \int_A \left( \lim_{T \to \infty} \frac{1}{T} \int_0^T g(f^tx) \, dt \right) \mu(dx) \quad \textrm{since $\mathrm{Basin}(\mu)$ is a $\mu$-full set} \\
    &= \lim_{T \to \infty} \int_A \left(\frac{1}{T} \int_0^T g(f^tx) \, dt \right) \mu(dx) \quad \textrm{by dominated convergence} \\
    &= \lim_{T \to \infty} \frac{1}{T} \int_0^T \left( \int_A g(f^tx) \, \mu(dx) \right) dt \quad \textrm{by Fubini's theorem} \\
    &= \lim_{T \to \infty} \frac{1}{T} \int_0^T \left( \int_{f^{-t}A} g(f^tx) \, \mu(dx) \right) dt \quad \textrm{since $f^{-t}A=A$} \\
    &= \lim_{T \to \infty} \frac{1}{T} \int_0^T \left( \int_A g(x) \, \mu(dx) \right) dt \quad \textrm{since $f^t\mu=\mu$} \\
    &= \int_A g(x) \, \mu(dx) \\
    &= \int_X g \, d\mu_2.
\end{align*}
Hence $\mu_1=\mu_2$, and so in particular, $\mu(A)^2 = \mu_1(A) = \mu_2(A) = \mu(A)$, and so $\mu(A) \in \{0,1\}$.
\end{proof}

Conceptually, a physical measure represents a more ``experimentally accessible'' version~\cite{Baladi:2002} of the kind of statistical dynamics depicted in the characterisation of ergodicity in Proposition~\ref{prop:ergequiv}. The definition is as follows.

\begin{defi}
A \emph{physical measure} is a compactly supported Borel probability measure $\mu$ on $X$ such that $\textrm{Basin}(\mu)$ includes $m$-almost all points in a neighbourhood of the support of $\mu$.
\end{defi}

\begin{rmk} \label{rmk:weaker_physical}
A commonly found weaker definition of a physical measure~\cite{Young:2002} is a compactly supported Borel probability measure $\mu$ on $X$ such that $m\big(\textrm{Basin}(\mu)\big)>0$. Since the original work of Sinai, Ruelle and Bowen, more general definitions of SRB measures have been formulated beyond the Axiom~A setting~\cite{Young:2002,Young:2016}, and it has been shown in such a broader setting that SRB measures are physical at least in this weaker sense~\cite{Pugh:1989}.
\end{rmk}

\subsection{Mixing measures and attracting measures}

Note that if $\mu$ and $\nu$ are Borel probability measures on $X$ such that $f^t\nu \to \mu$ weakly as $t \to \infty$, then $\mu$ is an invariant measure.

\begin{defi}
A \emph{mixing measure} is a Borel probability measure $\mu$ on $X$ such that the following equivalent statements hold:
\begin{enumerate}[\indent (i)]
    \item for all $A,B \in \mathcal{B}(X)$, $\mu(A \cap f^{-t}B) \to \mu(A)\mu(B) \textrm{ as } t \to \infty$;
    \item for every $\mu$-absolutely continuous probability measure $\tilde{\mu}$, we have $f^t\tilde{\mu} \to \mu$ strongly as $t \to \infty$;
    \item for every $\mu$-absolutely continuous probability measure $\tilde{\mu}$, we have $f^t\tilde{\mu} \to \mu$ weakly as $t \to \infty$.
\end{enumerate}
\end{defi}

It is well-known that every mixing measure is ergodic.

\begin{rmk}
Let us briefly describe how the non-trivial implications between (i), (ii) and (iii) are derived:
\begin{itemize}
    \item (ii) follows from (i) by Lemma~\ref{lemma:cone_conv} with $C=\{c\mathbbm{1}_A : c \geq 0, A \in \mathcal{B}(X)\}$ and $g=\mathbbm{1}_B$ for each $B \in \mathcal{B}(X)$, with the measures `$m$' and `$\mu$' both set to $\mu$.
    \item (i) follows from (iii) by the $\pi$-$\lambda$ theorem, using the facts that: the set
    \[ \Pi := \{B \in \mathcal{B}(X) \, : \, \mu(\partial B)=0\}\]
    is closed under pairwise intersections (since $\partial(A_1 \cap A_2) \subset (\partial A_1) \cup (\partial A_2)$) and $\Pi$ is a generator of $\mathcal{B}(X)$ (since it includes the open ball of all but countably many radii about each point in $X$); and if (iii) holds then $\Pi$ is contained in the set
    \[ \Lambda := \{B \in \mathcal{B}(X) \, : \, \forall \, A \in \mathcal{B}(X), \ \mu(A \cap f^{-t}B) \to \mu(A)\mu(B)\}, \]
    and $\Lambda$ is closed under countable disjoint unions by the discrete dominated convergence theorem since
    \[ \mu(A \cap f^{-t}B) \leq \mu(f^{-t}B) = \mu(B). \]
\end{itemize}
\end{rmk}

Conceptually, an attracting measure represents a more ``experimentally accessible'' version~\cite{Baladi:2002} of the kind of decay of correlations depicted in the notion of a mixing measure (see also Remark~\ref{rmk:corr}). The definition is as follows.

\begin{defi} \label{def:attracting}
An \emph{attracting measure} is a compactly supported Borel probability measure $\mu$ on $X$ whose support admits an open neighbourhood $U$ such that for every $m$-absolutely continuous probability measure $\nu$ with $\nu(U)=1$, we have that $f^t\nu \to \mu$ weakly as $t \to \infty$.
\end{defi}

\begin{rmk} \label{rmk:weak_attracting}
Along the lines of Remark~\ref{rmk:weaker_physical}, one could give a weaker definition of an attracting measure, where the set $U$ in Definition~\ref{def:attracting} is not required to be a neighbourhood of $\mathrm{supp}\,\mu$ but is allowed to be any $m$-positive measure set.
\end{rmk}

\begin{rmk} \label{rmk:corr}
For any finite measure $\nu$ on $X$, the set of bounded continuous functions from $X$ to $\mathbb{R}$ is a dense subset of $L^1(\nu,\mathbb{R})$~(\cite[Theorem~11.5]{Driver:2003}). Hence, by Lemma~\ref{lemma:cone_conv} (with $C$ being the set of nonnegative-valued bounded continuous functions) mixing measures and attracting measures can be understood in terms of \emph{correlations for continuous observables}:
\begin{itemize}
    \item A probability measure $\mu$ is mixing if and only if for all bounded continuous functions $g_1,g_2 \colon X \to \mathbb{R}$ the \emph{``classical'' correlation}~\cite{Baladi:2002}
    \[ \mathrm{Covar}_\mu[g_1, g_2 \circ f^t] =\! \int_X g_1(x) g_2(f^tx) \, \mu(dx) - \int_X g_1 \, d\mu \int_X g_2 \, d\mu \]
    tends to $0$ as $t \to \infty$.
    \item A probability measure $\mu$ with compact support $A$ is attracting if and only if there is an open neighbourhood $U$ of $A$ with $m(U)<\infty$ such that for all bounded continuous functions $g_1,g_2 \colon U \to \mathbb{R}$ the \emph{operational correlation}~\cite{Baladi:2002}
    \[ \int_U g_1(x) g_2(f^tx) \, m(dx) - \int_U g_1 \, dm \int_A g_2 \, d\mu \]
    tends to $0$ as $t \to \infty$.
\end{itemize}
\end{rmk}

\begin{rmk}
A possible reason why physical measures have received more attention than attracting measures is computational: If $\mu$ is a physical measure then for a typical initial condition $x$ near the support of $\mu$, we can numerically approximate $\mu$ just by simulating the single trajectory $f^tx$ up to a sufficiently large time, whereas by contrast, in order to use the attracting-measure property to numerically approximate an attracting measure $\mu$, we need to simulate a sufficiently large ensemble of trajectories up to a sufficiently large time. Interestingly, however, when we extend to nonautonomous settings such as in \cite{ashwin2021physical,newman23}, we can no longer use just a single trajectory to approximate a time-dependent physical measure at a given time; this is because the approximating measure is not concentrated on a set produced by an iteration rule of type $f_{n+1}(x) = \xi_n(f_n(x))$, but rather is concentrated on a set produced by an iteration rule of type $f_{n+1}(x)=f_n(\xi_n(x))$. (Switching the latter for the former would result in inefficiently approximating the physical measure of the past-limiting autonomous system, rather than approximating the physical measure of the nonautonomous system.)
\end{rmk}

\subsection{Physical and attracting measures supported on attractors}

\begin{defi}
A non-empty compact set $A \subset X$ is said to be \emph{invariant} if for all $t \geq 0$, $f^tA=A$.
\end{defi}

\begin{defi}
The \emph{basin} of a non-empty compact invariant set $A \subset X$ is given by
\[ \mathrm{Basin}(A) = \left\{ x \in X \, : \, d(f^tx,A) \to 0 \ \textrm{ as $t \to \infty$} \right\}. \]
\end{defi}

Now there are various non-equivalent definitions of an attractor; here, we will work with a fairly weak definition:

\begin{defi} \label{def:attractor}
An \emph{attractor} is a non-empty compact invariant set $A$ such that $\mathrm{Basin}(A)$ is a neighbourhood of $A$.
\end{defi}

\begin{prop}
If $A$ is an attractor then $\mathrm{Basin}(A)$ is open.
\end{prop}

\begin{proof}
Letting $V$ be an open neighbourhood of $A$ contained in $\mathrm{Basin}(A)$, we have that
\[ \mathrm{Basin}(A) = \bigcup_{t \geq 0} f^{-t}V. \qedhere \]
\end{proof}

For measures whose support is an attractor, one can define ``physical'' and ``attracting'' in terms of the basin of the support rather than just in terms of an arbitrary neighbourhood:

\begin{prop} \label{prop:basin}
Let $\mu$ be a Borel probability measure whose support is an attractor. Then
\begin{enumerate}[\indent (A)]
    \item $\mu$ is physical if and only if $\mathrm{Basin}(\mu)$ includes $m$-almost all points in $\mathrm{Basin}(\mathrm{supp}\,\mu)$;
    \item $\mu$ is attracting if and only if for every $m$-absolutely continuous probability measure $\nu$ with $\nu(\mathrm{Basin}(\mu))=1$, we have that $f^t\nu \to \mu$ weakly as $t \to \infty$.
\end{enumerate}
\end{prop}

\begin{proof}
The `if' directions are immediate from the definitions.

(A) Suppose $\mu$ is physical. Let $U$ be an open neighbourhood of $\mathrm{supp}\,\mu$ contained in $\mathrm{Basin}(\mathrm{supp}\,\mu)$ such that $m$-almost every point in $U$ is in $\mathrm{Basin}(\mu)$. So $U \setminus \mathrm{Basin}(\mu)$ is an $m$-null set, and so by Lemma~\ref{lemma:nonsing}, for each $n \geq 0$,
\[ f^{-n}\big( U \setminus \mathrm{Basin}(\mu) \big) \]
is also an $m$-null set. But it is not hard to show that $f^{-n}(\mathrm{Basin}(\mu))=\mathrm{Basin}(\mu)$ for each $n$, and $\mathrm{Basin}(\mathrm{supp}\,\mu) = \bigcup_{n=0}^\infty f^{-n}U$. So
\[ \mathrm{Basin}(\mathrm{supp}\,\mu) \setminus \mathrm{Basin}(\mu) = \bigcup_{n=0}^\infty f^{-n}\big( U \setminus \mathrm{Basin}(\mu) \big). \]
So $\mathrm{Basin}(\mathrm{supp}\,\mu) \setminus \mathrm{Basin}(\mu)$ is an $m$-null set.

(B) Suppose $\mu$ is attracting. Let $U$ be an open neighbourhood of $\mathrm{supp}\,\mu$ contained in $\mathrm{Basin}(\mathrm{supp}\,\mu)$ such that for every $m$-absolutely continuous probability measure $\nu$ with $\nu(U)=1$, we have that $f^t\nu \to \mu$ weakly as $t \to \infty$. For each $n \geq 0$, let $\mathcal{C}_n$ be the set of $m$-absolutely continuous probability measures $\nu$ with $\nu(f^{-n}U)=1$. For every $\nu \in \mathcal{C}_n$, we have $(f^n\nu)(U)=1$ and $f^n\nu$ is $m$-absolutely continuous (by Lemma~\ref{lemma:nonsing}), and so $f^{t+n}\nu=f^t(f^n\nu) \to \mu$ weakly as $t \to \infty$, i.e.\ $f^t\nu \to \mu$ weakly as $t \to \infty$. Now since $\mathrm{Basin}(\mathrm{supp}\,\mu) = \bigcup_{n=0}^\infty f^{-n}U$, we can apply Lemma~\ref{lemma:cone_conv} with
\[ C = \bigcup_{n=0}^\infty \big\{ c\tfrac{d\nu}{dm} \, : \, c \geq 0, \, \nu \in \mathcal{C}_n \big\} \]
to conclude that for every $m$-absolutely continuous probability measure $\nu$ with $\nu(\mathrm{Basin}(\mu))=1$, we have that $f^t\nu \to \mu$ weakly as $t \to \infty$.
\end{proof}

\subsection{Non-implications among properties of measures} \label{sec:nonergodic}

We have defined ergodic measures, physical measures, mixing measures, and attracting measures. As we have indicated, all of these are invariant measures, and all mixing measures are ergodic measures.

Regarding non-implications:
\begin{itemize}
    \item For a constant speed flow on the circle, the normalised Lebesgue measure is ergodic and physical but neither mixing nor attracting.
    \item A Dirac mass on a repelling fixed point is mixing but neither attracting nor physical.
    \item It is well-known that physical measures can fail to be ergodic~\cite{MunozYoung:2007}. The classic example consists of an attracting heteroclinic cycle between two hyperbolic fixed points $\mathbf{p}_1$ and $\mathbf{p}_2$ in the plane, so that trajectories attracted to the cycle spend asymptotically $\lambda$ proportion of the time near $\mathbf{p}_1$ and $1-\lambda$ proportion of the time near $\mathbf{p}_2$, for some constant $0<\lambda<1$. In this case $\lambda\delta_{\mathbf{p}_1}+(1-\lambda)\delta_{\mathbf{p}_2}$ is physical but not ergodic.
\end{itemize}

We now address non-implications of the attracting-measure property.

For convenience, define $\mathcal{F} \colon (1,\infty) \to (0,\infty)$ by $\mathcal{F}(t) = \frac{1}{t\log t}$; so $\mathcal{F}'(t)=-\frac{1+\log t}{t^2(\log t)^2}$, and $\mathcal{F}$ has a well-defined inverse function $\mathcal{F}^{-1} \colon (0,\infty) \to (1,\infty)$.

\begin{prop} \label{prop:counterex}
Suppose we have a compact Riemannian manifold $\tilde{X}$ and a $C^1$ vector field $\tilde{b}$ on $\tilde{X}$ whose flow $(\tilde{f}^t)_{t \in \mathbb{R}}$ admits a mixing measure $\tilde{\mu}$ that is equivalent to the Riemannian volume measure on $\tilde{X}$. Let $X=\tilde{X} \times \mathbb{R}$, let $\mu=\tilde{\mu} \otimes \delta_0$, and define
\[ b \colon X \to TX \cong (T\tilde{X}) \times \mathbb{R} \]
by
\[ b(x,y) = \begin{cases}
    (\mathbf{0},0) & y=0 \\
    \big(y\tilde{b}(x),\mathrm{sgn}(y)\mathcal{F}'(\mathcal{F}^{-1}(|y|))\big) & y \neq 0.
\end{cases} \]
Then $b$ is a $C^1$ vector field on $X$, and for the semiflow $(f^t)_{t \geq 0}$ on $X$ generated by $b$, the measure $\mu$ is attracting but neither physical nor ergodic (and hence also not mixing).
\end{prop}

\begin{proof}
Note that $b$ is well-defined, since for all $y \in \mathbb{R} \setminus \{0\}$, we have that $|y|$ is in the domain $(0,\infty)$ of $\mathcal{F}^{-1}$ and $\mathcal{F}^{-1}(|y|)$ is in the domain $(1,\infty)$ of $\mathcal{F}'$. Now the $T\tilde{X}$-component of $b(x,y)$ is equal to $y\tilde{b}(x)$ for all $(x,y) \in X$, and this has $C^1$ dependence on $(x,y)$; so in order to show that $b$ is a $C^1$ vector field, it only remains to show that the function $b_2 \colon \mathbb{R} \to \mathbb{R}$ given by
\[ b_2(y) = \begin{cases}
        0 & y=0 \\
        \mathrm{sgn}(y)\mathcal{F}'(\mathcal{F}^{-1}(|y|)) & y\neq 0
    \end{cases} \]
is $C^1$. Now $\mathcal{F}$ is $C^\infty$ on $(1,\infty)$ and $\mathcal{F}'$ is strictly negative on $(1,\infty)$; hence, by the inverse function theorem, we have that $\mathcal{F}^{-1}$ is $C^\infty$ on $(0,\infty)$. So $b_2$ is $C^\infty$ on $\mathbb{R} \setminus \{0\}$. It remains to show that $b_2$ is continuously differentiable at $0$. As $y \to 0$, $\mathcal{F}^{-1}(|y|) \to \infty$ and so $\mathcal{F}'(\mathcal{F}^{-1}(|y|)) \to 0$; hence $b_2$ is continuous at $0$. Therefore, in order to show that $b_2$ is continuously differentiable at $0$, it is sufficient (by the Mean Value Theorem) to show that $b_2'(y)$ is convergent as $y \to 0$. For $y \in \mathbb{R} \setminus \{0\}$, we have
\[ b_2'(y) = \frac{\mathcal{F}''(\mathcal{F}^{-1}(|y|))}{\mathcal{F}'(\mathcal{F}^{-1}(|y|))}. \]
Now $\mathcal{F}''(t)=\frac{2(\log t)^2 + 3\log t + 2}{t^3(\log t)^3}$. Note that for all sufficiently large $t$, $\mathcal{F}''(t)<\frac{1}{t^3}$ and $|\mathcal{F}'(t)|>\frac{1}{t^{2+\alpha}}$ for $\alpha \in (0,1)$, and hence $\left|\frac{\mathcal{F}''(t)}{\mathcal{F}'(t)}\right|<t^{\alpha-1}$. So $\frac{\mathcal{F}''(t)}{\mathcal{F}'(t)} \to 0$ as $t \to \infty$. Hence, as $y \to 0$, $\mathcal{F}^{-1}(|y|) \to \infty$ and so $b_2'(y) \to 0$.

Thus we have shown that $b$ is a $C^1$ vector field. One can verify that $(f^t)_{t \geq 0}$ is explicitly given by
\[ f^t(x,y) = \begin{cases}
    (x,0) & y=0 \\
    \big( \tilde{f}^{\psi_y(t)}x , \varphi_y(t) \big) & y \neq 0
\end{cases} \]
where
\begin{align*}
    \varphi_y(t) &= \mathrm{sgn}(y)\mathcal{F}\big(t + \mathcal{F}^{-1}(|y|)\big) \\
    \psi_y(t) &= \mathrm{sgn}(y)\Big(\log\log\!\big(t + \mathcal{F}^{-1}(|y|)\big) - \log\log\!\big(\mathcal{F}^{-1}(|y|)\big)\Big).
\end{align*}

We first show that $\mu$ is attracting with respect to $(f^t)$. Let $\tilde{m}$ be the Riemannian measure on $\tilde{X}$, and let $m$ be the product of $\tilde{m}$ and the Lebesgue measure on $\mathbb{R}$. Let $U=\tilde{X} \times (-1,1) \subset X$. So $U$ is a neighbourhood of $\mathrm{supp}\,\mu$. Now fix any $m$-absolutely continuous probability measure $\nu$ on $X$ with $\nu(U)=1$, and let $h=\frac{d\nu}{dm}$. We have $\int_{-1}^1 \int_{\tilde{X}} h(x,y) \, \tilde{m}(dx) \, dy = 1$, and so in particular, for Lebesgue-almost all $y \in (-1,1)$, $\int_{\tilde{X}} h(x,y) \, \tilde{m}(dx)<\infty$. Hence, for Lebesgue-almost every $y \in (-1,1)$, we can define an $\tilde{m}$-absolutely continuous finite measure $\nu_y$ on $\tilde{X}$ by
\[ \nu_y(S) = \int_S h(x,y) \, \tilde{m}(dx) \quad \forall \, S \in \mathcal{B}(\tilde{X}). \]
Now $\nu_y$ is $\tilde{\mu}$-absolutely continuous, and so $\tilde{f}^t\nu_y$ converges weakly to $\nu_y(\tilde{X})\tilde{\mu}$ as $t \to \pm\infty$. For each $t$, we have
\[ f^t\nu = \int_{-1}^1 (\tilde{f}^{\psi_y(t)}\nu_y) \otimes \delta_{\varphi_y(t)} \, dy, \]
and considering the large-$t$ behaviour of the integrand: we have $\varphi_y(t) \to 0$ and $|\psi_y(t)| \to \infty$, and so for Lebesgue-almost all $y$,
\[ (\tilde{f}^{\psi_y(t)}\nu_y) \otimes \delta_{\varphi_y(t)} \, \overset{\textrm{weakly}}{\to} \, \big(\nu_y(\tilde{X})\tilde{\mu}\big) \otimes \delta_0 \, = \, \nu_y(\tilde{X})\mu \]
as $t \to \infty$. We can then apply the dominated convergence theorem, since $\int_{-1}^1 \nu_y(\tilde{X}) \, dy = 1 < \infty$, to obtain that $f^t\nu$ converges weakly to $\mu$ as $t \to \infty$.

It is clear that $\mu$ is not ergodic, since every point in $\mathrm{supp}\,\mu$ is a fixed point of $(f^t)$ but $\mu$ itself is not a Dirac mass.

Finally, we show that $\mu$ is not physical, by showing that every $(x,y) \in X$ with $y \neq 0$ is not in $\mathrm{Basin}(\mu)$. Fix $(x,y) \in X$ with $y > 0$; an identical argument holds for $y<0$. Let
\[ p_T = \frac{1}{T} \int_0^T \delta_{\tilde{f}^{\psi_y(t)}x} \, dt \]
for each $T>0$. Note that $p_T$ is the $\tilde{X}$-marginal of $\frac{1}{T} \int_0^T \delta_{f^t(x,y)} \, dt$. So, since $\tilde{\mu}$ is not a Dirac mass, it will be sufficient to show that ``$p_T$ tends towards being Dirac as $T \to \infty$''. Specifically, we will show that for all $\varepsilon>0$, $p_T(\tilde{B}_\varepsilon(\tilde{f}^{\psi_y(T)}x)) \to 1$ as $T \to \infty$, where $\tilde{B}_\varepsilon(\cdot)$ denotes the open $\varepsilon$-ball in the geodesic distance $\tilde{d}$ on $\tilde{X}$.

For convenience, let $c=\mathcal{F}^{-1}(y)$. Note that $\psi_y \colon (0,\infty) \to (0,\infty)$ is increasing and bijective. Let $M=\max_{z \in \tilde{X}} \|b(z)\|$. Now fix $\varepsilon>0$. For each $T>0$ with $\psi_y(T) \geq \frac{\varepsilon}{M}$, for each $\tau \in (\psi_y(T)-\frac{\varepsilon}{M},\psi_y(T)]$, we have $\tilde{d}(\tilde{f}^\tau x,\tilde{f}^{\psi_y(T)}x)<\varepsilon$; so
\begin{align*}
   1 - p_T(\tilde{B}_\varepsilon(\tilde{f}^{\psi_y(T)}x)) &= \frac{\mathrm{Lebesgue}\{0 \leq t \leq T : \tilde{d}(\tilde{f}^{\psi_y(t)} x,\tilde{f}^{\psi_y(T)}x) \geq \varepsilon \}}{T} \\
   &\leq \frac{\mathrm{Lebesgue}\{0 \leq t \leq T : \psi_y(t) \leq \psi_y(T)-\frac{\varepsilon}{M} \}}{T} \\
   &= \frac{\psi_y^{-1}(\psi_y(T)-\frac{\varepsilon}{M})}{T}.
\end{align*}
Now $\psi_y^{-1}(\psi_y(T)-\frac{\varepsilon}{M})$ is the unique value $t$ such that
\[ \log\log(t+c) = \log\log(T+c) - \tfrac{\varepsilon}{M}, \]
and solving this gives
\[ \psi_y^{-1}(\psi_y(T)-\tfrac{\varepsilon}{M}) = (T+c)^{e^{-\frac{\varepsilon}{M}}} - c = \Big(1+\frac{c}{T}\Big)^{e^{-\frac{\varepsilon}{M}}}T^{e^{-\frac{\varepsilon}{M}}} - c. \]
So
\[ 1 - p_T(\tilde{B}_\varepsilon(\tilde{f}^{\psi_y(T)}x)) \leq \frac{\psi_y^{-1}(\psi_y(T)-\frac{\varepsilon}{M})}{T} = \Big(1+\frac{c}{T}\Big)^{e^{-\frac{\varepsilon}{M}}}T^{-1+e^{-\frac{\varepsilon}{M}}} - \frac{c}{T} \ \to 0 \ \textrm{ as } T \to \infty. \qedhere \]
\end{proof}

\begin{rmk}
In the proof of Proposition~\ref{prop:counterex}, the only properties of $\varphi_y$ and $\psi_y$ that were needed in order to obtain that $\mu$ is attracting were that $\varphi_y(t) \to 0$ and $|\psi_y(t)| \to \infty$ as $t \to \infty$; but to obtain that $\mu$ is not physical, our strategy relied on a particularly slow growth rate of $\psi_y$. If, for example, we had chosen $\mathcal{F}(t)=\frac{1}{t}$ so that $b(x,y)=(y\tilde{b}(x),-|y|y)$ and $\psi_y$ has logarithmic growth $|\psi_y(t)|=\log(t+c)-\log(c)$, then we would still have that $\mu$ is attracting and non-ergodic, but the proof of non-physicality would not have gone through, since our bound on $1 - p_T(\tilde{B}_\varepsilon(\tilde{f}^{\psi_y(T)}x))$ would have been
\[ \frac{1}{T}\left[(T+c) \!\times\! (e^{-\frac{\varepsilon}{M}}) - c \right] \]
rather than
\[ \frac{1}{T}\left[(T+c)^{\,\wedge\,} (e^{-\frac{\varepsilon}{M}}) - c \right] \]
and the former tends to $e^{-\frac{\varepsilon}{M}} \approx 1$ rather than to $0$ as $T \to \infty$.
\end{rmk}

\section{From mixing to attracting} \label{sec:Result}

In \cite{Bowen:1975}, the SRB measure on an Axiom~A attractor is constructed in Theorem~3.3; then, \cite[Theorem~5.1]{Bowen:1975} establishes that the SRB measure is physical, while \cite[Theorem~5.3]{Bowen:1975} establishes that the SRB measure is attracting under a mild extra assumption. The proofs of these two facts are separately built upon the material developed prior to Sec.~5 of \cite{Bowen:1975}. More precisely, the mild extra condition is that the unstable manifold at each point in the attractor is dense in the attractor; but this only enters the proof via the fact that it implies that the SRB measure is mixing (\cite[Remark~3.5]{Bowen:1975}). So what the proof of \cite[Theorem~5.3]{Bowen:1975} really says is that if the SRB measure on an Axiom~A attractor is mixing then it is attracting; in this section, we provide a topological abstraction of the overarching strategy adopted in \cite{Bowen:1975} to show this.

We first need to introduce a particular type of attractor.

\subsection{Attractors via orbit-tracking}

\begin{defi}
Given a set $U \subset X$, a non-empty compact invariant set $A \subset X$, and a value $\varepsilon>0$, we say that \emph{$U$ $\varepsilon$-orbit-tracks $A$} if there exists $T \geq 0$ such that for every $x \in U$ there exists $y \in A$ such that for all $t \geq T$, $d(f^tx,f^ty)<\varepsilon$.
\end{defi}

Note that in the above definition, given $\varepsilon$, the value of $T$ is not free to be chosen in terms of $x$ as an unbounded function of $x$, but rather must be uniform across $x \in U$.

\begin{defi}
We will say that a non-empty compact invariant set $A \subset X$ is \emph{stable via orbit-tracking} if for every $\varepsilon>0$, there is a neighbourhood $U_\varepsilon$ of $A$ that $\varepsilon$-orbit-tracks $A$.
\end{defi}

\begin{defi} \label{def:orbtr}
We will say that a non-empty compact invariant set $A \subset X$ is an \emph{attractor via orbit-tracking} if there is a neighbourhood $U$ of $A$ with the property that for every $\varepsilon>0$, $U$ $\varepsilon$-orbit-tracks $A$.
\end{defi}

We now show how being an attractor via orbit-tracking can be characterised in terms of being an attractor that is stable via orbit-tracking.

\begin{prop}
Let $A \subset X$ be a non-empty compact invariant set.
\begin{enumerate}[\indent (A)]
    \item If $A$ is an attractor via orbit-tracking, then $A$ is stable via orbit-tracking and is an attractor (in the sense of Definition~\ref{def:attractor}): any set $U$ fulfilling the description in Definition~\ref{def:orbtr} is contained in $\mathrm{Basin}(A)$.
    \item Suppose $X$ is locally compact. If $A$ is stable via orbit-tracking and is an attractor (in the sense of Definition~\ref{def:attractor}), then $A$ is an attractor via orbit-tracking: any compact neighbourhood of $A$ contained in $\mathrm{Basin}(A)$ fulfils the description of $U$ in Definition~\ref{def:orbtr}.
\end{enumerate}
\end{prop}

\begin{proof}
(A) is immediate. For (B), fix a compact neighbourhood $U$ of $A$ contained in $\mathrm{Basin}(A)$, and fix $\varepsilon>0$. Take a neighbourhood $\tilde{U}$ of $A$ and a value $\tilde{T} \geq 0$ such that for every $x \in \tilde{U}$ there exists $y \in A$ such that for all $t \geq \tilde{T}$, $d(f^tx,f^ty)<\varepsilon$. For every $x \in U$, since $x$ is in $\mathrm{Basin}(A)$, we can find $t$ such that $f^tx \in \tilde{U}^\circ$, which immediately implies that $f^tV \subset \tilde{U}$ for some neighbourhood $V$ of $x$. Consequently, by the compactness of $U$, there exist sets $V_1,\ldots,V_n \subset X$ and $t_1,\ldots,t_n \geq 0$ (for some $n \in \mathbb{N}$) such that $U \subset \bigcup_{i=1}^n V_i$ and $f^{t_i}V_i \subset \tilde{U}$ for each $i$. Now let $T=\tilde{T}+\max(t_1,\ldots,t_n)$. For each $x \in U$, taking $i$ such that $x \in V_i$, we have $f^{t_i}x \in \tilde{U}$ and so we can find $\tilde{y} \in A$ such that for all $t \geq \tilde{T}$, $d(f^{t+t_i}x,f^t\tilde{y})<\varepsilon$. Since $A$ is invariant, we can then find $y \in A$ such that $\tilde{y}=f^{t_i}y$. So for all $t \geq T \geq \tilde{T}+t_i$, we have $d(f^tx,f^ty)<\varepsilon$.
\end{proof}

Note that for a general invariant set, being an attractor via orbit-tracking is strictly stronger than being stable via orbit-tracking. For example, for the flow on $\mathbb{R}^2$ generated by $(\dot{x},\dot{y})=(-y,x)$, every circle centred on $(0,0)$ is stable via orbit-tracking but is not an attractor via orbit-tracking.

As implied by the last sentence of the proof of \cite[Proposition~4.4]{Bowen:1975} (which is itself based on \cite[Theorem~7.4]{Hirsch:1970}), Axiom~A attractors are attractors via orbit-tracking.

\subsection{Abstraction of Bowen and Ruelle's result}

The strategy of \cite{Bowen:1975} for proving that mixing SRB measures on Axiom~A attractors are attracting is to evolve the initial absolutely continuous distribution $\nu$ sufficiently far forward in time under $(f^t)$ that $f^t\nu$ can be approximated by $\mu$-absolutely continuous probability measures, and then to choose this approximation so as to remain a good approximation forever afterwards under the subsequent time-evolution by $(f^t)$; so since $\mu$-absolutely continuous measures evolve under $(f^t)$ towards $\mu$ itself (by the definition of mixing), $\nu$ also evolves under $(f^t)$ towards $\mu$. The ability to choose the approximation so as to remain good under indefinite subsequent time-evolution by $(f^t)$ is based on the facts that Axiom~A attractors are attractors via orbit-tracking, and SRB measures on Axiom~A attractors have a kind of ``controlled chaoticness relative to the Riemannian measure'' obtained in \cite[Corollary~4.6]{Bowen:1975}.

The topological essence of the above strategy is communicated more precisely in our following result, which applies to the general setup of Sec.~\ref{sec:general_setup}.

\begin{thm} \label{thm:mixSRB}
Let $\mu$ be a mixing measure whose support $A$ is an attractor via orbit-tracking. Suppose furthermore that there exist arbitrarily small values $\delta>0$ for which one can find an unbounded set $\mathcal{T} \subset [0,\infty)$ with
\begin{equation} \label{eq:CEhyp} \inf_{\tau \in \mathcal{T}\!, \, x \in A} \, \frac{\mu \!\left( y \in A \, : \ d(f^tx,f^ty) < \delta \ \ \forall \, t \in [0,\tau] \right)}{m \!\left( y \in X \, : \ d(f^tx,f^ty) < 3\delta \ \ \forall \, t \in [0,\tau] \right)} \ > \ 0. \end{equation}
Then $\mu$ is attracting.
\end{thm}

\begin{rmk}
In \cite{Bowen:1975}, $2\delta$ is used where we have $3\delta$ above. More specifically, \cite[Corollary~4.6]{Bowen:1975} says that the SRB measure on an Axiom~A attractor fulfils \eqref{eq:CEhyp} with $2\delta$ in place of $3\delta$ for all sufficiently small $\delta>0$, with $\mathcal{T}$ being the whole of $[0,\infty)$; and then \cite{Bowen:1975} uses this fact to prove \cite[Theorem~5.3]{Bowen:1975}. In this proof of \cite[Theorem~5.3]{Bowen:1975}, the use of the $2\delta$ is connected with the use of a maximal $(d_\tau,\delta)$-separated subset of $A$ under the metric $d_\tau(x,y)=\max_{0 \leq t \leq \tau} d(f^tx,f^ty)$. (A set is called $(d_\tau,\delta)$-separated if no two distinct points in the set are within a $d_\tau$-distance less than $\delta$ of each other.) However, the proof does not then seem to go through correctly, unless one can establish a finite $\tau$-independent upper bound on the number of mutually intersecting $(d_\tau,\delta)$-balls around the points in this maximal $(d_\tau,\delta)$-separated subset of $A$. It is not clear that one can do this. Therefore, in our proof of Theorem~\ref{thm:mixSRB}, we consider a maximal disjoint union of $(d_\tau,\delta)$-balls around points in $A$ rather than a maximal $(d_\tau,\delta)$-separated subset of $A$; in order to follow the logic of the proof of \cite[Theorem~5.3]{Bowen:1975}, this then requires the $2\delta$ to be strengthened to $3\delta$.
\end{rmk}

\subsection{Some general preliminaries for the proof of Theorem~\ref{thm:mixSRB}}

\subsubsection{Bi-separated sets}

We state the definition and result below with reference to the metric $d$, but in the proof of Theorem~\ref{thm:mixSRB} this will be applied with reference to a different, topologically equivalent metric.

\begin{defi}
Given $\delta>0$, we say that a set $\mathcal{S} \subset X$ is \emph{$(d,\delta)$-bi-separated} if for any two distinct $x,y \in \mathcal{S}$, $B_\delta(x) \cap B_\delta(y)=\emptyset$.
\end{defi}

Note that if a $(d,\delta)$-bi-separated set $\mathcal{S}$ is contained in a compact subset of $X$, then $\mathcal{S}$ must be finite.

\begin{lemma} \label{lemma:bisep}
Given a compact set $A \subset X$ and a value $\delta>0$, there exists a subset $\mathcal{S}$ of $A$ such that $\mathcal{S}$ is $(d,\delta)$-bi-separated and $B_\delta(A) \subset B_{3\delta}(\mathcal{S})$.
\end{lemma}

\begin{proof}
Since $A$ is compact, there must exist a $(d,\delta)$-bi-separated set $\mathcal{S} \subset A$ such that for every $x \in A \setminus \mathcal{S}$, the set $\mathcal{S} \cup \{x\}$ is not $(d,\delta)$-bi-separated. So for any $x \in A$, $B_\delta(x)$ has non-empty intersection with $B_\delta(\mathcal{S})$, and so by the triangle inequality $x \in B_{2\delta}(\mathcal{S})$. So we have shown that $A \subset B_{2\delta}(\mathcal{S})$, and so (again by the triangle inequality) $B_\delta(A) \subset B_{3\delta}(\mathcal{S})$.
\end{proof}

\subsubsection{Approximation of non-absolutely continuous measures by absolutely continuous measures}

We describe a general procedure that can enable an arbitrary probability measure $\nu$ supported near $\mathrm{supp}\,\mu$ to be approximated by a $\mu$-absolutely continuous probability measure $P$.

\begin{defi} \label{def:singular_to_ac}
Let $\mu$ be a measure on $X$, and let $\nu$ be a probability measure on $X$. Suppose that for some finite set $\mathcal{S}$ we have a family $\mathcal{M}=(M_i)_{i \in \mathcal{S}}$ of sets $M_i \in \mathcal{B}(X)$ with $0<\mu(M_i)<\infty$ and a pairwise-disjoint family $\mathcal{N}=(N_i)_{i \in \mathcal{S}}$ of mutually disjoint sets $N_i \in \mathcal{B}(X)$ with $\sum_{i \in \mathcal{S}} \nu(N_i) = 1$. Then we define the measure $P[\nu;\mu,\mathcal{M},\mathcal{N}]$ on $X$ by
\[ P[\nu;\mu,\mathcal{M},\mathcal{N}] \ = \ \sum_{i \in \mathcal{S}} \nu(N_i)\mu(\,\cdot\,|M_i), \]
where $\mu(B|M_i):=\frac{\mu(B \cap M_i)}{\mu(M_i)}$ for all $B \in \mathcal{B}(X)$.
\end{defi}

Note that since $\sum_{i \in \mathcal{S}} \nu(N_i) = 1$, the measure $P[\nu;\mu,\mathcal{M},\mathcal{N}]$ is a convex combination of probability measures $\mu(\,\cdot\,|M_i)$ and thus is itself a probability measure. Furthermore, it is clear that $P[\nu;\mu,\mathcal{M},\mathcal{N}]$ is $\mu$-absolutely continuous. The following basic fact can be understood intuitively as saying that ``if $\max_{i \in \mathcal{S}} \mathrm{diam}(M_i \cup N_i)$ is small then $P[\nu;\mu,\mathcal{M},\mathcal{N}]$ is close to $\nu$''.

\begin{lemma} \label{lemma:ac_approximation}
Fix $\varepsilon>0$ and a bounded measurable function $g \colon X \to \mathbb{R}$. In the setting of Definition~\ref{def:singular_to_ac}, suppose that for each $i \in \mathcal{S}$ there exists $x_i \in M_i$ such that for all $y \in M_i \cup N_i$, $|g(y)-g(x_i)|<\frac{\varepsilon}{2}$. Then
\[ \left| \int_X g \, d\nu - \int_X g \, dP[\nu;\mu,\mathcal{M},\mathcal{N}] \right| \ < \ \varepsilon. \]
\end{lemma}

\begin{proof}
We have
\begin{align*}
    &\phantom{{}={}} \ \Bigg| \int_X g \, d\nu - \int_X g \, dP[\nu;\mu,\mathcal{M},\mathcal{N}] \Bigg| \\
    &= \ \left| \sum_{i \in \mathcal{S}} \left( \int_{N_i} g \, d\nu \ - \ \nu(N_i)\int_X g(y) \, \mu(dy|M_i) \right) \right| \\
    &= \ \left| \sum_{i \in \mathcal{S}} \left( \int_{N_i} g(y) - g(x_i) \, \nu(dy) \ + \ \nu(N_i)\int_X g(x_i) - g(y) \, \mu(dy|M_i) \right) \right| \\
    &\leq \ \sum_{i \in \mathcal{S}} \left( \int_{N_i} |g(y) - g(x_i)| \, \nu(dy) + \nu(N_i)\int_X |g(x_i) - g(y)| \, \mu(dy|M_i) \right) \\
    & < \ \tfrac{\varepsilon}{2}\sum_{i \in \mathcal{S}} (\nu(N_i) + \nu(N_i)) \\
    & = \ \varepsilon. \qedhere
\end{align*}
\end{proof}

Let us point out that Lemma~\ref{lemma:ac_approximation} does not require the collection of sets $\mathcal{M}$ to be mutually disjoint, although for our proof of Theorem~\ref{thm:mixSRB} it will be.

\subsubsection{Dominated weak convergence of absolutely continuous measures}

\begin{lemma} \label{lemma:abscont}
Let $\mu$ be a Borel measure on $X$ and let $(h_n)$ be a sequence of measurable functions $h_n \colon X \to [0,\infty]$ that is dominated by an $n$-independent $\mu$-integrable function $H \colon X \to [0,\infty]$. Suppose the sequence of measures $\mu_n \colon S \mapsto \int_S h_n \, d\mu$ converges weakly to a measure $\mu_\infty$. Then $\mu_\infty$ is $\mu$-absolutely continuous.
\end{lemma}

\begin{proof}
Let $N$ be any $\mu$-null set. We will show that $\mu_\infty(N)<\varepsilon$ for all $\varepsilon>0$. Fix $\varepsilon$. Since finite measures on metric spaces are regular~\cite[Theorem~1.2]{Parthasarathy:2005} and the measure $S \mapsto \int_S H \, d\mu$ is a finite measure assigning zero measure to $N$, we can find an open neighbourhood $U$ of $N$ such that $\int_U H \, d\mu < \varepsilon$, and hence $\mu_\infty(N) \leq \mu_\infty(U) \leq \liminf_{n \to \infty} \mu_n(U) < \varepsilon$.
\end{proof}

\subsection{Proof of Theorem~\ref{thm:mixSRB}}

For each $\tau>0$, define the metric $d_\tau(x,y)=\max_{0 \leq t \leq \tau} d(f^tx,f^ty)$. Note that $d_\tau$ is topologically equivalent to $d$. For each $\delta>0$ and $x \in X$, let $B_\delta^\tau(x)$ be the open $d_\tau$-ball of radius $r$ about $x$. So if, for some $\varepsilon>0$, a set $U \subset X$ $\varepsilon$-orbit-tracks $A$, then there exists $T \geq 0$ such that for all $\tau>0$, $f^TU \subset B_\varepsilon^\tau(A)$.

Now since $A$ is compact and a finite union of sets satisfying (H) satisfies (H), there exists a neighbourhood of $A$ that satisfies (H). So let $U$ be an open neighbourhood of $A$ satisfying (H) such that $U$ $\varepsilon$-orbit-tracks $A$ for every $\varepsilon>0$. We will show that for every $m$-absolutely continuous probability measure $\nu$ with $\nu(U)=1$, we have $f^t\nu \to \mu$ weakly as $t \to \infty$. Since $L^\infty\big(U,m|_{\mathcal{B}(U)}\big) \cap L^1\big(U,m|_{\mathcal{B}(U)}\big)$ is dense in $L^1\big(U,m|_{\mathcal{B}(U)}\big)$, it is sufficient by Lemma~\ref{lemma:cone_conv} just to consider $\nu$ with bounded density.

So fix an $m$-absolutely continuous probability measure $\nu$ of bounded density such that $\nu(U)=1$, and fix a bounded $1$-Lipschitz function $g \colon X \to \mathbb{R}$. Fix any $\varepsilon>0$; we will show that $\limsup_{t \to \infty} \left| \int_X g \, d(f^t\nu) - \int_X g \, d\mu \right| \ \leq \ \varepsilon.$

Take $\delta \in (0,\frac{\varepsilon}{6})$ such that \eqref{eq:CEhyp} holds for some unbounded set $\mathcal{T}$, which we can take to be a countable set $\mathcal{T}=\{\tau_n : n \geq 1\}$ without loss of generality. Let $T \geq 0$ be such that for all $\tau>0$, $f^TU \subset B_\delta^\tau(A)$. For each $n \geq 1$, on the basis of Lemma~\ref{lemma:bisep} let $\mathcal{S}_n$ be a $(d_{\tau_n},\delta)$-bi-separated subset of $A$ such that $B_\delta^{\tau_n}(A) \subset B_{3\delta}^{\tau_n}(\mathcal{S}_n)$. So in particular, $f^TU \subset B_{3\delta}^{\tau_n}(\mathcal{S}_n)$. Let $\mathcal{M}_n:=(B_\delta^{\tau_n}(x))_{x \in \mathcal{S}_n}$. Let $\mathcal{N}_n\!=\!(N_{x,n})_{x \in \mathcal{S}_n}$ be a pairwise-disjoint cover of $f^TU$ by Borel sets $N_{x,n} \subset B_{3\delta}^{\tau_n}(x)$. Since $\nu(U)=1$ and $f^TU \subset \bigcup_{x \in \mathcal{S}_n} N_{x,n}$, we have that $\bigcup_{x \in \mathcal{S}_n} N_{x,n}$ is a $f^T\nu$-full measure set. Let $\mu_n:=P[f^T\nu;\mu,\mathcal{M}_n,\mathcal{N}_n]$. Since $g$ is $1$-Lipschitz and $3\delta<\frac{\varepsilon}{2}$, for each $t \in [0,\tau_n]$, we have $|g(f^ty) - g(f^tx)|<\frac{\varepsilon}{2}$ for all $x \in \mathcal{S}_n$ and $y \in B_\delta^{\tau_n}(x) \cup N_{x,n}$, and hence Lemma~\ref{lemma:ac_approximation} applied to $g \circ f^t$ gives
\begin{equation} \label{eq:mu_n} \left| \int_X g \, d(f^{T+t}\nu) - \int_X g \, d(f^t\mu_n) \right| < \varepsilon. \end{equation}
Now since the measures $\mu_n$ are $\mu$-absolutely continuous and hence supported on the compact set $A$, we can find a weak limit $\mu_\infty$ of a subsequence $(\mu_{m_n})$ of $(\mu_n)$. If there existed $t \geq 0$ such that
\[ \left| \int_X g \, d(f^{T+t}\nu) - \int_X g \, d(f^t\mu_\infty) \right| \ > \ \varepsilon \]
then all sufficiently large $n$ would have that
\[ \left| \int_X g \, d(f^{T+t}\nu) - \int_X g \, d(f^t\mu_{m_n}) \right| \ > \ \varepsilon, \]
contradicting \eqref{eq:mu_n} for the infinitely many values of $n$ with $\tau_n \geq t$. Hence we have that for all $t \geq 0$,
\begin{equation} \label{eq:muinfty} \left| \int_X g \, d(f^{T+t}\nu) - \int_X g \, d(f^t\mu_\infty) \right| \ \leq \ \varepsilon. \end{equation}
Now for all $n$,
\begin{align*}
\frac{d\mu_n}{d\mu} \ &= \ \sum_{x \in \mathcal{S}_n} \frac{f^T\nu(N_{x,n})}{\mu(B_\delta^{\tau_n}(x))}\mathbbm{1}_{B_\delta^{\tau_n}(x)} \\
&\leq \ \max_{x \in \mathcal{S}_n} \frac{f^T\nu(N_{x,n})}{\mu(B_\delta^{\tau_n}(x))} \qquad \textrm{since $\mathcal{S}_n$ is $(d_{\tau_n},\delta)$-bi-separated} \\
&\leq \ \max_{x \in \mathcal{S}_n} \frac{f^T\nu(B_{3\delta}^{\tau_n}(x))}{\mu(B_\delta^{\tau_n}(x))} \\
&\leq \ \max_{x \in \mathcal{S}_n} \frac{m(B_{3\delta}^{\tau_n}(x))}{\mu(B_\delta^{\tau_n}(x))} \left\|\frac{d(f^T\nu)}{dm}\right\|_\infty \\
&\leq \ C \left\|\frac{d(f^T\nu)}{dm}\right\|_\infty
\end{align*}
where $C$ can be taken independent of $n$ (since $\delta$ and $\mathcal{T}$ fulfil \eqref{eq:CEhyp}), and $\left\|\frac{d(f^T\nu)}{dm}\right\|_\infty$ is finite by Lemma~\ref{lemma:bd}. Hence by Lemma~\ref{lemma:abscont}, we have that $\mu_\infty$ is $\mu$-absolutely continuous, and so since $\mu$ is mixing,
\[ \left| \int_X g \, d(f^t\mu_\infty) - \int_X g \, d\mu \right| \,\to\, 0 \ \textrm{ as } t \to \infty \, ; \]
and combining this with \eqref{eq:muinfty} gives
\[ \limsup_{t \to \infty} \left| \int_X g \, d(f^{T+t}\nu) - \int_X g \, d\mu \right| \ \leq \ \varepsilon. \]

\section{Outlook} \label{sec:outlook}

A natural future direction is the development of methods for establishing the existence of attracting measures on attractors outside the Axiom~A setting of \cite{Bowen:1975,Ruelle:1976}, such as non-uniformly hyperbolic attractors. One basic question in this direction is whether there are cases beyond the Axiom~A setting where Theorem~\ref{thm:mixSRB} can be applied.

This direction has implications not only for autonomous dynamical systems: The physical-measure property and the attracting-measure property are relevant for nonautonomous and random dynamical systems~\cite{Baladi:2002,ashwin2021physical,newman23}; and in the setting of nonautonomous dynamical systems with an autonomous past limit, sufficient conditions have been provided in \cite{newman23} to be able to construct an attracting measure of the nonautonomous system given the knowledge of an attracting measure of the past-limiting autonomous system. So the natural step then needed to ``complete the picture'' is a way to establish the existence of an attractor measure for the past-limiting autonomous system.

Another natural question is whether, in analogy to the result of \cite{Pugh:1989} mentioned in Remark~\ref{rmk:weaker_physical}, one can find conditions under which mixing measures that are SRB in the sense of \cite{Young:2002} are guaranteed to be attracting at least in the weak sense of Remark~\ref{rmk:weak_attracting}.

\subsection*{Acknowledgments}

We would like to thank Ian Melbourne for useful discussions. For the purpose of open access, the authors have applied a Creative Commons Attribution (CC BY) licence to any Author Accepted Manuscript version arising. Funded by the European Union's Horizon 2020 research and innovation programme under grant agreement No 820970 (TiPES).

\bibliographystyle{plain}

%{\footnotesize 
\bibliography{attrrefs}
%}

\end{document}